\documentclass[11pt]{article}
\usepackage{amsmath, amssymb, amsthm}
\usepackage{verbatim}
\usepackage{multicol}
\usepackage{enumerate}
\usepackage{comment}
\usepackage{dsfont}
\usepackage[none]{hyphenat}
\usepackage[unicode]{hyperref}
\hypersetup{
	colorlinks=true,
	linkcolor=blue,
	filecolor=magenta,      
	urlcolor=cyan,
	citecolor=blue
}
\usepackage{pgf}
\usepackage{tikz}
\usetikzlibrary{positioning,arrows,shapes,decorations.markings,decorations.pathreplacing,matrix,patterns}
\tikzstyle{vertex}=[circle,draw=black,fill=black,inner sep=0,minimum size=3pt,text=white,font=\footnotesize]
\usepackage{cleveref}

\date{}
\title{\vspace{-1.2cm} Matrix discrepancy and the log-rank conjecture}

\author{ Benny Sudakov\thanks{ETH Zurich, \emph{e-mail}: \textbf{benjamin.sudakov@math.ethz.ch}.
Research supported in part by SNSF grant 200021\_196965.}, 
	Istv\'an Tomon\thanks{Ume\r{a} University, \emph{e-mail}: \textbf{istvan.tomon@umu.se}.}}

\theoremstyle{plain}
\newtheorem{theorem}{Theorem}[section]

\newtheorem{corollary}[theorem]{Corollary}
\newtheorem{claim}[theorem]{Claim}
\newtheorem{lemma}[theorem]{Lemma}

\Crefname{theorem}{Theorem}{Theorems}
\Crefname{definition}{Definition}{Definitions}
\Crefname{corollary}{Corollary}{Corollaries}
\Crefname{claim}{Claim}{Claims}
\Crefname{lemma}{Lemma}{Lemmas}
\Crefname{conjecture}{Conjecture}{Conjectures}
\Crefname{problem}{Problem}{Problems}
\Crefname{prop}{Proposition}{Propositions}

\theoremstyle{definition}

\DeclareMathOperator{\rank}{rank}
\DeclareMathOperator{\polylog}{polylog}

\DeclareMathOperator{\disc}{disc}

\DeclareMathOperator{\pdisc}{pdisc}

\begin{document}
	
	\maketitle
	\sloppy

 \begin{abstract}
     Given an $m\times n$ binary matrix $M$ with $|M|=p\cdot mn$ (where $|M|$ denotes the number of 1 entries), define the \emph{discrepancy} of $M$ as $\disc(M)=\displaystyle\max_{X\subset [m], Y\subset [n]}\big||M[X\times Y]|-p|X|\cdot |Y|\big|$.
     Using semidefinite programming and spectral techniques, we prove that if $\rank(M)\leq r$ and $p\leq 1/2$, then
      $$\disc(M)\geq \Omega(mn)\cdot \min\left\{p,\frac{p^{1/2}}{\sqrt{r}}\right\}.$$
    We use this result to obtain a modest improvement of Lovett's best known upper bound on the log-rank conjecture.
    We prove that any $m\times n$ binary matrix $M$ of rank at most $r$ contains an $(m\cdot 2^{-O(\sqrt{r})})\times (n\cdot 2^{-O(\sqrt{r})})$ sized all-1 or all-0 submatrix, which implies that the deterministic communication complexity of any Boolean function of rank $r$ is at most $O(\sqrt{r})$.
 \end{abstract}

\section{Introduction}

The log-rank conjecture, proposed by Lov\'asz and Saks \cite{logrank} in 1988, is one of the fundamental open problems in communication complexity. It states that for any Boolean function $f:X\times Y\rightarrow \{-1,1\}$ of rank $r$, its deterministic communication complexity $\mbox{CC}^{det}(f)$ is bounded by $\polylog(r)$. We refer the reader to Kushilevitz and Nisan \cite{KN} for exact definitions, or the recent survey of Lee and Shraibman \cite{LS} for a detailed overview of the problem. The log-rank conjecture has a number of combinatorial interpretations, which will be the main focus of our paper. 

Lov\'asz and Saks \cite{logrank}  observed  (see also \cite{Lovett}) that this problem is closely related to bounding the chromatic number of graphs, whose adjacency matrix has bounded rank. Indeed, if $c(r)$ denotes the maximum of $\mbox{CC}^{det}(f)$ among every $f$ of rank $r$, then $\chi(G)\leq 2^{c(r+1)}$ for every graph $G$, whose adjacency matrix has rank at most $r$. Furthermore, the log-rank conjecture is related to finding monochromatic submatrices in low-rank binary matrices. To this end, let $\alpha(r)>0$ be such that any binary matrix $M\in \{0,1\}^{m\times n}$ of rank at most $r$ contains an all-0 or all-1 submatrix of area at least $mn/2^{\alpha(r)}$ (here, the \emph{area} of a matrix refers to the product of the number of rows and columns). It was proved by Nisan and Wigderson \cite{NW} that $c(f)=O((\log r)^2+\sum_{i=0}^{\log_2 r}\alpha(r/2^i))$.

The best known lower bound $c(r)=\Tilde{\Omega}((\log r)^2)$ is due to G\"o\"os, Pitassi, and Watson \cite{GPW}. On the other hand, the best known upper bound was due to Lovett \cite{Lovett}, who showed that $\alpha(r)\leq O(\sqrt{r}\cdot \log r)$, and thus $c(r)=O(\sqrt{r}\cdot\log r)$. Here, we provide the modest improvement that $c(r)=O(\sqrt{r})$ by establishing the following improvement on $\alpha(r)$.

\begin{theorem}\label{thm:logrank}
    There exists a constant $c>0$ such that the following holds. Let $M\in \{0,1\}^{m\times n}$ such that $\rank(M)\leq r$. Then there exists $X\subset [m]$ and $Y\subset [n]$ such that 
    $$|X|\geq \frac{m}{2^{c\sqrt{r}}}\mbox{ and }|Y|\geq \frac{n}{2^{c\sqrt{r}}},$$
    and $M[X\times Y]$  contains only 0 or only 1 entries.
\end{theorem}

One of the key components of the proof of Lovett \cite{Lovett} is a result of Linial, Mendelson, Schechtman, and Shraibman \cite{LMSS} (see also \cite{LS}) on the \emph{discrepancy} of low-rank matrices. We examine an alternative definition of discrepancy, inspired by the one commonly used in graph theory \cite{EGPS,RST}. This notion of discrepancy provides stronger bounds for sparse matrices, which allows us to remove  the log-factor from Lovett's bound. Given a binary matrix $M\in \{0,1\}^{m\times n}$, let $|M|$ denote the number of 1 entries of $M$. Writing $p=\frac{|M|}{mn}$, define the \emph{discrepancy} of $M$ as
$$\disc(M)=\max_{X\subset [m], Y\subset [n]}\big||M[X\times Y]|-p|X|\cdot |Y|\big|.$$
Viewing $M$ as the bipartite graph, whose vertex classes are the rows and columns, and whose edges correspond to the 1 entries of $M$, this definition of discrepancy is a bipartite analogue of the graph discrepancy 
introduced in the 80's by Erd\H{o}s, Goldberg, Pach, and Spencer \cite{EGPS} and extensively studied since then. Our main contribution is the following lower bound on the discrepancy of low-rank matrices.

\begin{theorem}\label{thm:disc_main}
    There exists a constant $c>0$ such that the following holds. Let $M\in \{0,1\}^{m\times n}$ such that $\rank(M)\leq r$. If $p=\frac{|M|}{mn}\leq 1/2$, then $\disc(M)\geq cmn\cdot \min\left\{p,\frac{p^{1/2}}{\sqrt{r}}\right\}$.
\end{theorem}

Note that this bound is the best possible. In case $p\gg 1/r$, i.e. when the minimum is attained by $p^{1/2}/\sqrt{r}$, consider a random $r\times r$ binary matrix $R$ in which each entry is $1$ with probability $p$. It is easy to show that  $\disc(R)=\Theta(\sqrt{p}r^{3/2})$ with high probability. For every $m,n$, then one can construct the $m\times n$ matrix $M$ by repeating each row of $R$ $m/r$-times, and repeating each column $n/r$-times. Then $\rank(M)\leq r$ and $\disc(M)=\frac{mn}{r^2}\cdot\disc(R)=\Theta(mn p^{1/2}/\sqrt{r})$. On the other hand, if $p\ll 1/r$, then the minimum is attained by $p$. Trivially, $|M|$ is always an upper bound on the discrepancy, so our theorem is tight in this case as well.

In order to prove Theorem \ref{thm:disc_main}, we follow the recent framework of R\"aty, Sudakov, and Tomon \cite{RST}. While their results are concerned about the usual graph theoretic notion of discrepancy, working in a bipartite setting simplifies many of their ideas. We consider a semidefinitie relaxation of the discrepancy, and then apply spectral techniques to prove our lower bounds. We remark that this approach is fairly different from (and perhaps more elementary than)  that of Linial, Mendelson, Schechtman, and Schraibman \cite{LS}, which relies on John's ellipsoid theorem \cite{ellipsoid}.

 \section{Discrepancy}

In this section, we present basic notions and results about matrices and discrepancy, and introduce a semidefinite relaxation. We omit floors and ceilings whenever they are not crucial.

An $m\times n$ binary matrix $M$ naturally corresponds to the bipartite graph, whose two vertex classes are identified with $[m]=\{1,\dots,m\}$ and $[n]$, and there is an edge between $i$ and $j$ if $M_{i,j}=1$. While we will use the language of binary matrices, this correspondence is good to keep in mind for many arguments.

Given  $M\in \{0,1\}^{m\times n}$, let $p(M)=\frac{|M|}{mn}$ be the \emph{density} of $M$, and let $d(M)=\frac{2|M|}{m+n}$  (which is the average degree of the corresponding bipartite graph). In what follows, fix some $M\in \{0,1\}^{m\times n}$ and write $p=p(M)$. Given $X\subset [m]$ and $Y\subset [n]$, $M[X\times Y]$ is the submatrix of $M$ induced by the rows $X$ and columns $Y$. Define
$$\disc_{M}(X,Y)=\disc(X,Y)=|M[X,Y]|-p|X||Y|.$$
Furthermore, define the \emph{positive discrepancy} of $M$ as $$\disc^{+}(M)=\max_{X\subset [m],Y\subset [n]} \disc(X,Y)$$ and  the \emph{negative discrepancy} as $$\disc^{-}(M)=\max_{X\subset [m],Y\subset [n]} -\disc(X,Y).$$
Then $\disc(M)=\max\{\disc^{+}(M),\disc^{-}(M)\}$. Note that as $\disc(\emptyset,\emptyset)=0$, we have $\disc^{+}(G)\geq 0$ and $\disc^{-}(G)\geq 0$. Also, if  $X$ and $X'$ are disjoint, then $\disc(X\cup X',Y)=\disc(X,Y)+\disc(X',Y)$. Next, we show that $\disc^{-}(M)$ and $\disc^{+}(M)$ cannot be too different.

\begin{claim}\label{claim:pos_neg}
$\disc^{+}(M)=\Theta(\disc^{-}(M))$.
\end{claim}
\begin{proof}
 Let $X\subset [m]$ and $Y\subset [n]$ such that $\disc^{+}(M)=\disc(X,Y).$ Let $X'=[m]\setminus X$ and $Y'=[n]\setminus Y$. Then, $0=\disc([m],[n])=\disc(X,Y)+\disc(X',Y)+\disc(X,Y')+\disc(X',Y')$. Hence, at least  one of the terms is at most $-\disc(X,Y)/3$, which shows that $\disc^{-}(M)\geq \disc^{+}(M)/3$. Similarly, $\disc^{+}(M)\geq \disc^{-}(M)/3$.
\end{proof}

Next, we show that it is enough to consider submatrices of $M$ that intersect exactly half of the rows and half of the columns to get the discrepancy up to a constant.

\begin{claim}\label{claim:half}
   There exists $X_0\subset [m]$ and $Y_0\subset [n]$ such that $|X_0|=m/2$ and $|Y_0|=n/2$, and $$\disc(X_0,Y_0) \leq -\disc^{-}(M)/12.$$
\end{claim}

\begin{proof}
Let $X\subset [m]$ and $Y\subset [n]$ such that $|\disc(X,Y)|=\disc(M)$. Assume that $\disc(X,Y)\geq 0$, the other case can be handled similarly. 

First, let us assume that $|X|\leq m/2$. Let $X_1$ be a random $m/2$ element subset of $[m]$ containing $X$. Then 
\begin{align*}
    \mathbb{E}(\disc(X_1,Y))&=\disc(X,Y)+\mathbb{E}(\disc(X_1\setminus X,Y))
    =\disc(X,Y)+\frac{m/2-|X|}{m-|X|}\disc(A\setminus X,Y)\\
    &\geq \disc(X,Y)- \frac{m/2-|X|}{m-|X|}\disc(X,Y)\geq \frac{1}{2}\disc(X,Y).
\end{align*}
Here, the first inequality holds by the maximality of $(X,Y)$. Therefore, there exists a choice for $X_1$ such that $\disc(X_1,Y)\geq \frac{1}{2}\disc(X,Y)$. In case $|X|\geq m/2$, let $X_1$ be a random $m/2$ element subset of $X$. Similarly as before, by considering the expectation, there exists a choice for $X_1$ such  that $\disc(X_1,Y)\geq \frac{1}{2}\disc(X,Y)$.

Now let $X_2\subset [m]$ and $Y_1\subset [n]$ such that $|X_2|=m/2$, and $|\disc(X_2,Y_1)|$ is maximal under these conditions. Assume that $\disc(X_2,Y_1)\geq 0$, the other case can be handled similarly. Then $\disc(X_2,Y_1)\geq \disc(X_2,Y)$. If $|Y_1|\leq n/2$, let $Y_2$ be a random $n/2$ element subset of $[n]$ that contains $Y_1$. As before, by considering the expectation of $\disc(X_2,Y_2)$, we get that there is a choice for $Y_2$ such that $\disc(X_2,Y_2)\geq \frac{1}{2}\disc(X_2,Y_1)\geq \frac{1}{4}\disc(X,Y)$. If $|Y_1|\geq n/2$, then we can proceed as above by taking $Y_2$ to be a random $m/2$ element subset of $Y_1$.

In conclusion, we get $X_2\subset [m]$ and $Y_2\subset [n]$ such that $|X_2|=m/2$, $|Y_2|=n/2$, and $|\disc(X_2,Y_2)|\geq \frac{1}{4}|\disc(X,Y)|$. If $\disc(X_2,Y_2)\leq 0$, we are done, so we may assume that $\disc(X_2,Y_2)>0$.  Let $X_3=[m]\setminus X_2$, and $Y_3=[n]\setminus Y_2$. Writing the equality $0=\disc(X_2,Y_2)+\disc(X_3,Y_2)+\disc(X_2,Y_3)+\disc(X_3,Y_3)$, at least one of the terms is at most $-\disc(X_2,Y_2)/3\leq -|\disc(X,Y)|/12$. Hence, we can choose $X_0$ and $Y_0$ such that $|X_0|=n/2$, $|Y_0|=m/2$ and $\disc(X_0,Y_0)\leq -\disc(M)/12\leq -\disc^{-}(M)/12$.
\end{proof}

Consider the following relaxation of the positive discrepancy. Given $\mathbf{x}\in \mathbb{R}^m$ and $\mathbf{y}\in \mathbb{R}^n$, let
$$\disc(\mathbf{x},\mathbf{y})=\mathbf{x}^T(M-pJ)\mathbf{y},$$
where $J$ is the $m\times n$ sized all-1 matrix. Also, define $$\disc_0^{+}(M)=\max_{\mathbf{x}\in [-1,1]^m,\mathbf{y}\in [-1,1]^n} \disc(\mathbf{x},\mathbf{y}).$$ 

\begin{claim}
$\disc^{+}(M)=\Theta(\disc_0^{+}(M)).$
\end{claim}
\begin{proof}
Let $\mathbf{x}\in [-1,1]^m,\mathbf{y}\in [-1,1]^n$ be such that $\disc_0^{+}(M)=\disc(\mathbf{x},\mathbf{y})$. As $\disc(\mathbf{x},\mathbf{y})$ is a linear function in every variable, we can find $\mathbf{x}\in \{-1,1\}^m$ and $\mathbf{y}\in \{-1,1\}^n$ achieving the maximum. Let $X_1=\{i\in [m]:\mathbf{x}(i)=1\}$, $X_2=[m]\setminus X_1$, $Y_1=\{i\in [n]:\mathbf{y}(i)=1\}$ and $Y_2=[n]\setminus Y_1$. Then
$$\disc(\mathbf{x},\mathbf{y})=\disc(X_1,Y_1)-\disc(X_1,Y_2)-\disc(X_2,Y_1)+\disc(X_2,Y_2).$$
Hence, at least one of the four terms has absolute value at least $\frac{1}{4}\disc_0^+(M)$. But this shows that either $\disc^{+}(M)$ or $\disc^{-}(M)$ is at least $\frac{1}{4}\disc_0^+(M)$, and we are done by Claim \ref{claim:pos_neg}.
\end{proof}

Let us further relax the definition of discrepancy by assigning vector values to the rows and columns. That is, let
$$\pdisc_0(M)=\max \sum_{i,j} M_{i,j}\langle v_i,w_j\rangle -p\sum_{i,j}\langle v_i,w_j\rangle,$$
where the maximum is taken over all $v_1,\dots,v_m,w_1,\dots, w_{n}\in \mathbb{R}^{m+n}$ such that $||v_i||_2,||w_j||_2\leq 1$. Clearly, $\pdisc_0(M)\geq \disc^{+}(M)$. However, it follows from Grothendieck's inequality \cite{Groth} that $\pdisc_0(M)=O(\disc^{+}(M))$ also holds.

	\begin{lemma}[Grothendieck's inequality]\label{lemma:groth}
		There exists a universal constant $K>0$ such that the following holds. For a matrix $Q\in \mathbb{R}^{n\times m}$ let
		$\beta=\sup\left\{\sum_{i,j} Q_{i,j}x_iy_j: x_i, y_j \in [-1,1]\right\}$
		and
		$\beta^{*}=\sup\left\{\sum_{i,j} Q_{i,j}\langle v_i,w_j\rangle : v_i, w_j,\in \mathbb{R}^{n+m},||v_i||_2\leq 1, ||w_i||_2\leq 1\right\}$.
		Then $\beta\leq \beta^{*}\leq K\beta$.
	\end{lemma}

\begin{claim}
    $\pdisc_0(M)=\Theta(\disc^{+}(M))=\Theta(\disc^{-}(M))$
\end{claim}

\begin{proof}
Apply Groethendieck's inequality with the matrix $Q\in \mathbb{R}^{m\times n}$ defined as $Q_{i,j}=M_{i,j}-p$. 
\end{proof}

Let $N=m+n$. Let $A\in\mathbb{R}^{N\times N}$ be the adjacency matrix of the bipartite graph corresponding to $M$, and let $L$ be the adjacency matrix of the complete bipartite graph on the same vertex set. Formally, $A$ is the symmetric matrix defined as $A_{i,j+m}=A_{j+m,i}=M_{i,j}$ for $(i,j)\in [m]\times [n]$, and $A_{i,j}=0$ for $(i,j)\in [m]\times [m]$ and $(i,j)\in [m+1,N]\times [m+1,N]$. 
Note that $\rank(A)=2\rank(M)$.
On the other hand, we can write  $L=\frac{1}{2}(e\cdot e^{T}-f\cdot f^{T})$, where $e$ is the all 1 vector, and $f\in \mathbb{R}^N$ is defined $f(i)=1$ if $i\in [m]$ and $f(i)=-1$ if $i\in [m+1,N]$. Call $A$ the \emph{symmetrization} of $M$.

Given a matrix $X\in \mathbb{R}^{N\times N}$, define its discrepancy as
$$\disc_M(X)=\disc(X)=\langle X,A\rangle-p\langle X,L\rangle,$$
where $\langle .,.\rangle$ denotes the Frobenius (i.e. entry-wise) inner product. Furthermore, define 
$$\pdisc(M)=\max \left\{\disc(X): X\mbox{ is symmetric, positive semidefinite, }\forall i\in [N], X_{i,i}\leq 1\right\}.$$ Then $\pdisc(M)=2\pdisc_{0}(M)$, as every $X$ satisfying the conditions in the definition of $\pdisc(M)$ corresponds to the Gramm-matrix of vectors $v_1,\dots,v_m,w_1,\dots,w_n$ of length at most $1$. In conclusion, we arrive to the following.

\begin{claim}
    $\pdisc(M)=\Theta(\disc^{+}(M))=\Theta(\disc^{-}(M)).$
\end{claim}

Finally, let us argue that considering only those matrices with equal number of rows and columns does not restrict us too much. Let $M^{\otimes}$ be the $(mn)\times (mn)$ sized matrix we get by repeating each row of $M$ $n$-times, and then repeating every column $m$ times. It is easy to see that $$\frac{1}{(mn)^2}\disc^{+}_0(M^{\otimes})=\frac{1}{mn}\disc^{+}_0(M),$$ 
and $\rank(M)=\rank(M^{\otimes})$.

\section{Low rank matrices}

Now let us turn to bounding $\pdisc(M)$. Let $A$ be the adjacency matrix of the bipartite graph corresponding to $M$, defined above. Let $\lambda_1\geq \dots\geq \lambda_{N}$ be the eigenvalues of $A$, and let $v_1,\dots,v_{N}$ be a corresponding orthonormal set of eigenvectors. Note that as $A$ is the adjacency matrix of a bipartite graph, we have $\lambda_i=-\lambda_{N+1-i}$ for $i\in [N]$, and the vectors $v_i(j)$ and $v_{N+1-i}(j)$ agree on $[m]$, and are opposite on $[m+1,N]$, see Chapter 11 of \cite{lovasz}. Formally, $v_i(j)=f(j)\cdot v_{N+1-i}(j)$ for every $i,j\in [N]$, where $f(j)=1$ if $j \in [m]$ and $-1$ if 
$j \in [m+1, N]$. Note that $\lambda_1\geq \dots\geq \lambda_n\geq 0\geq \lambda_{n+1}\geq\dots\geq \lambda_{N}$.

\begin{lemma}\label{lemma:discX}
Let $a_1,\dots,a_{N}\geq 0$ and let $X=\sum_{i=1}^{N}a_i\cdot  v_i\cdot v_i^{T}$. Then
$$\disc(X)\geq \sum_{i=1}^{N}a_i\lambda_i-\frac{pN}{2}\cdot\max_{i}(a_i+a_{N+1-i}).$$
\end{lemma}

\begin{proof}
Write $A=\sum_{i=1}^{N}\lambda_i v_i\cdot v_i^{T}$, and observe the identity $\langle v\cdot v^{T},w\cdot w^{T}\rangle=\langle v,w\rangle^2$. Then
$$\langle X,A\rangle=\sum_{i=1}^{N}\sum_{j=1}^{N}a_i\lambda_j \langle v_i,v_j\rangle^2=\sum_{i=1}^{N}\lambda_ia_i.$$
One the other hand,
$$\langle X,L\rangle =\frac{1}{2}\sum_{i=1}^{N}a_i(\langle v_i,e\rangle^2+\langle v_i,f\rangle^2)=\frac{1}{2}\sum_{i=1}^{N}(a_i+a_{N+1-i})\langle v_i,e\rangle^2\leq \frac{N}{2}\max_i (a_i+a_{N+1-i}).$$
Here, the second equality holds by noting that $\langle v_i,e\rangle=\langle v_{N+1-i},f\rangle$, and the last inequality as $\sum_{i=1}^{N}\langle v_i,e\rangle^2=||e||_2^2=N.$
\end{proof}

In what follows,  let us assume that $m=n$, so $N=2n$. Let  $d=d(M)=\frac{|M|}{n}=\frac{pN}{2}$, which is the average degree of the corresponding bipartite graph $G$, and let $\Delta=\Delta(M)$ denote the maximum number of 1 entries in a row or a column, that is, the maximum degree of $G$.

\begin{lemma}\label{lemma:cubesum}
$\pdisc(M)\geq \frac{1}{\Delta}\sum_{i=2}^{n} \lambda_i^3.$
\end{lemma}

\begin{proof}
Let $X=\frac{1}{\Delta}\sum_{i=1}^{n}\lambda_i^2\cdot v_i\cdot v_i^{T}$, which is positive semidefinite. We also claim that $X_{i,i}\leq 1$ for every $i\in [N]$. Indeed, note that
$\frac{1}{\Delta}A^2-X=\frac{1}{\Delta}\sum_{i=n+1}^{N}\lambda_i^2\cdot v_i\cdot v_i^{T}$ is positive semidefinite. Therefore all diagonal entries of 
$\frac{1}{\Delta}A^2-X$ are non-negative. Moreover, the $i$-th diagonal entry of $A^2$ counts the number of 1 in the $i$-th row of $A$. Therefore every diagonal entry of $\frac{1}{\Delta}A^2$ is at most $1$. Thus, $X_{i,i}\leq 1$ for every $i\in [N]$, which implies that $\pdisc(M)\geq \disc(X)$. Then, by Lemma \ref{lemma:discX} (with $a_i=\lambda_i^2$ for $1 \leq i \leq n$ and $a_i=0$ for $n+1 \leq i \leq N$), we have
$$\pdisc(M)\geq \disc(X)\geq \frac{1}{\Delta}\left(\sum_{i=1}^{n}\lambda_i^3-\frac{pN}{2} \cdot \lambda_1^2\right)=
\frac{1}{\Delta}\left(\sum_{i=1}^{n}\lambda_i^3-d\lambda_1^2\right)\geq \frac{1}{\Delta}\sum_{i=2}^{n}\lambda_i^3.$$
In the last inequality, we used the well known result that the maximal eigenvalue $\lambda_1$ of a graph $G$ is always at least its average degree $d$, see e.g. \cite{lovasz}.
\end{proof}

Next, we show that if $M$ has low rank and not too large maximum degree, then it has large discrepancy.

\begin{lemma}\label{lemma:lowrank}
Let $\rank(M)\leq r$, $d\leq n/2$ and $\Delta\leq 1.1 d$. Then $\pdisc(M)\geq \frac{d^{1/2}n^{3/2}}{7\sqrt{r}}$.
\end{lemma}

\begin{proof}
Note that $\sum_{i=1}^{2n}\lambda_i^2=\mbox{tr}(A)=2dn$, so $\sum_{i=2}^{n}\lambda_i^2\geq dn-\Delta^2\geq dn -1.21d^2\geq dn/3$. Here, we used that the maximum degree $\Delta$ of a graph is always at least the largest eigenvalue $\lambda_1$, see e.g. \cite{lovasz}. Moreover, as $\rank(M)\leq r$, we have $\lambda_{r+1}=\dots=\lambda_n=0$, so $\sum_{i=2}^{r}\lambda_i^2\geq dn/3$. Applying the inequality between the quadratic and cubic mean, we get
$$\left(\frac{\sum_{i=2}^{r}\lambda_i^3}{r-1}\right)^{1/3}\geq \left(\frac{\sum_{i=2}^{r}\lambda_i^2}{r-1}\right)^{1/2}\geq \frac{d^{1/2}n^{1/2}}{\sqrt{3(r-1)}}.$$
Hence, applying Lemma \ref{lemma:cubesum}, $\pdisc(M)\geq \frac{1}{1.1 d}\sum_{i=2}^{r}\lambda_i^3\geq \frac{d^{1/2}n^{3/2}}{7\sqrt{r}}$.
\end{proof}

We improve the previous result by showing that the maximum degree condition can be omitted.

\begin{lemma}\label{lemma:main}
Let $\rank(M)\leq r$, $d\leq n/2$. Then $\pdisc(M)\geq c\cdot \min\left\{dn,\frac{d^{1/2}n^{3/2}}{\sqrt{r}}\right\}$ for some $c>0$. 
\end{lemma}

\begin{proof}
 Let $D=\min\left\{dn,\frac{d^{1/2}n^{3/2}}{7\sqrt{r}}\right\}$ and $\delta=0.01$. For $i\in [n]$, write $\deg_r(i)$ for the number of 1 entries in row $i$, and let $\deg_c(i)$ be the number of 1 entries in column $i$. Let $U_r=\{i\in [n]: \deg_r(i)\geq (1+\delta)d\}$, and similarly, let $U_c=\{i\in [n]: \deg_c(i)\geq (1+\delta)d\}$. Finally, write $t_r$ for the number of 1 entries in $M[U_r\times [n]]$, and define $t_c$ analogously.  We have 
$$\disc(U_r,[n])=t_r-p|U_r|n=\sum_{i\in U_r} (\deg_r(i)-d)\geq \sum_{i\in U_r} \frac{\delta}{2}\cdot \deg_r(i)=\frac{\delta}{2} t_r.$$ Similarly, $\disc([n],U_c)\geq \frac{\delta}{2}t_c$. Therefore, if $t_r+t_c\geq 0.01D$, then $\disc^{+}(M)\geq \frac{\delta}{2}\cdot \max\{t_r,t_c\} \geq \frac{1}{4} 10^{-4} D$. Hence, as $\pdisc(M)=\Theta(\disc^{+}(M))$, we may assume that $t_r+t_c< 0.01D$. 

Change all 1 entries of $M$ that are contained in either $U_r\times [n]$ or $[n]\times U_c$ into a 0 entry, and let $M'$ be the resulting matrix. Then 
$$|M'|\geq |M|-t_r-t_c\geq |M|-0.01D\geq |M|-0.01dn,$$
so $d'=d(M')\geq 0.99d$ and $d'\leq d$. Furthermore, $\Delta(M')\leq (1+\delta)d\leq 1.1d'$, and $\rank(M')\leq r$. Therefore, we can apply Lemma \ref{lemma:lowrank} to conclude that $\pdisc(M')\geq \frac{d'^{1/2}n^{3/2}}{7\sqrt{r}}\geq D/2$. Let $X$ be matrix such that $\disc_{M'}(X)=\pdisc(M')$, and let $A'$ be the symmetrization of $M'$. Then
\begin{align*}
|\disc_{M}(X)-\disc_{M'}(X)|&=|\langle X,A-A'\rangle-\frac{d-d'}{n}\langle X,L\rangle|\\
&\leq \langle J,A-A'\rangle + 2(d-d')n\leq 4(|M|-|M'|)\leq 0.04D.
\end{align*}
Hence, $\disc_{M}(X)\geq \disc_{M'}(X)-0.04D\geq D/2-0.04D\geq D/4$,  finishing the proof.
\end{proof}

Note that the previous lemma immediately implies Theorem \ref{thm:disc_main} by considering $M^{\otimes}$ instead of $M$, and using that $\pdisc(M)=\Theta(\disc(M))$. It will be convenient to rewrite Lemma \ref{lemma:main} into a slightly more convenient form.

\begin{corollary}\label{cor:main}
There exists $c>0$ such that the following holds. Assume that $\rank(M)\leq r$ and  $1/(8r)\leq p\leq 1/2$. Then there exists $X,Y\subset [n]$  such that $|X|=|Y|=n/2$, and 
$$p(M[X\times Y])\leq p-c\frac{p^{1/2}}{\sqrt{r}}.$$ 
\end{corollary}

\begin{proof}
Note that $p=d/n$, so for $1/(8r)\leq p$, we have $\min\{dn,\frac{d^{1/2}n^{3/2}}{\sqrt{r}}\}=\Omega(p^{1/2}n^2/\sqrt{r})$. Therefore, $\disc^{-}(M)=\Omega(\pdisc(M))=\Omega(p^{1/2}n^2/\sqrt{r})$ by Lemma \ref{lemma:main}. By Claim \ref{claim:half}, there exist $X\subset [n]$ and $Y\subset [n]$ such that $|X|=|Y|=n/2$ and $-\disc(X,Y)=\Omega(\disc^{-}(M))$. But then if $M'=M[X\times Y]$, we have $$|M'|=p|X||Y|+\disc(X,Y)\leq p|X||Y|-\Omega(p^{1/2}n^2/\sqrt{r}),$$
which implies
$$p(M')=\frac{|M'|}{|X||Y|}\leq p-\Omega\left(\frac{p^{1/2}}{\sqrt{r}}\right).$$
\end{proof}

\section{Log-rank conjecture}

Let $z(M)$ denote the largest $z$ for which there exist $X\subset [n]$, $Y\subset [n]$, such that $|X|=|Y|=z$ and $M[X\times Y]$ has only 0 entries. Based on a lemma of Gavinsky and Lovett \cite{GL}, we show that sufficiently sparse low-rank matrices $M$ satisfy $z(M)=\Omega(n)$.

\begin{lemma}\label{lemma:toosparse}
Let $M\in\{0,1\}^{n\times n}$ such that $\rank(M)\leq r$ and $p(M)\leq \frac{1}{8r}$. Then $z(M)\geq n/4$.
\end{lemma}

\begin{proof}
Let $X,Y\subset [n]$ be the set of rows and columns containing more than $n/(4r)$ entries equal to 1, respectively. As $p(M)\leq \frac{1}{8r}$, we have $|X|,|Y|\leq n/2$. Let $X'=[n]\setminus X$ and $Y'=[n]\setminus Y$. Assume that $z(M[X'\times Y'])< n/4$, then we find a permutation matrix of size $r+1$ in $M[X'\times Y']$, contradicting that $\rank(M)\leq r$.

For $k=1,\dots,r+1$, we find a $k\times k$ sized permutation matrix greedily. Suppose that we have already found $A\subset X'$ and $B\subset Y'$ such that $|A|=|B|=k$, and $M[A\times B]$ is a permutation matrix. Let $X_0\subset X'$ be the set of all rows that intersect a column in $B$ in a 1 entry, and let $Y_0\subset Y'$ be the set of all columns that intersect a row in $A$ in a 1 entry. Then $|X_0|,|Y_0|\leq kn/(4r)\leq n/4$, so $|X'\setminus X_0|\geq n/4$ and $|Y'\setminus Y_0|\geq n/4$. Since we assumed that $z(M[X'\times Y'])< n/4$, there exists $i\in X'\setminus X_0$ and $j\in Y'\setminus Y_0$ such that $M_{i,j}=1$, which means that $(A\cup\{i\})\times (B\cup\{j\})$ induces  a $(k+1)\times (k+1)$ sized permutation matrix in $M$.
\end{proof}

\begin{theorem}
There exists a constant $c>0$ such that the following holds. Let $M\in \{0,1\}^{n\times n}$ such that $\rank(M)\leq r$ and $p(M)\leq 1/2$. Then $z(M)\geq n/2^{c\sqrt{r}}$.
\end{theorem}

\begin{proof}
We will proceed by a density decrement argument. Let $n_0=n$ and $M_0=M$, and let $c_0$ be the constant given by Corollary \ref{cor:main}. We define a sequence $M_0,M_1,\dots$ of submatrices of $M$ with decreasing density. If $M_{i}\in \{0,1\}^{n_i\times n_i}$ is already defined for $i\geq 0$ with $p_i:=p(M_i)\leq p(M)$, and we have $p_i\geq 1/(8r)$, we define $M_{i+1}$ by setting $M_{i+1}=M_i[X\cup Y]$, where $X,Y\subset [n_i]$,  $|X|=|Y|=n_i/2$, and 
$$p(M_i[X\times Y])\leq p_i-c_0\frac{p_i^{1/2}}{\sqrt{r}}.$$
Such $X$ and $Y$ exist by Corollary \ref{cor:main}. On the other hand, if $p(M_i)<1/(8r)$, then set $I=i$ and stop. 

Observe that $n_i=n/2^i$ and $p_{i+1}\leq p_i-c_0\frac{p_i^{1/2}}{\sqrt{r}}$ for $i=0,\dots,I-1$. This implies that if $C$ is a sufficiently large constant, then for every $x\in [1/(8r),1/2]$, the number of indices $i$ such that $x\leq p_i\leq 2x$ is at most $C\sqrt{r}x^{1/2}$. Applying this for every $x=2^{-k}$, we get that
$$I\leq C\sqrt{r}\sum_{i=1}^{\infty}2^{-i/2}\leq 10C\sqrt{r}.$$ 

As $p(M_I)\leq 1/(8r)$ and $n_I\geq n/2^{I}\geq n/2^{10C\sqrt{r}}$, we have $z(M_I)\geq n_{I}/2=n/2^{c\sqrt{r}}$  by Lemma \ref{lemma:toosparse} with a suitable constant $c>0$, finishing the proof.
\end{proof}

From this, the proof of Theorem \ref{thm:logrank} is immediate.

\begin{proof}[Proof of Theorem \ref{thm:logrank}]
Let $M\in \{0,1\}^{m\times n}$ such that $\rank(M)\leq r$. Without loss of generality, we may assume that $p(M)\leq 1/2$, otherwise we can consider $J-M$. Note that $\rank(J-M)\leq r+1$. Writing $n'=mn$,  $M^{\otimes}\in \{0,1\}^{n'\times n'}$ satisfies that $p(M^{\otimes})\leq 1/2$. Hence, by the previous theorem, $z(M^{\otimes})\geq n'/2^{c\sqrt{r}}$ for some constant $c>0$. Note that each row of $M$ is repeated $n$ times in $M^{\otimes}$, and each column is at most $m$ times, so we can find $X\subset [m]$ and $Y\subset [n]$ such that $|X|\geq z(M^{\otimes})/n$, $|Y|\geq z(M^{\otimes})/m$, and $M[X\times Y]$ has only 0 entries. But then $|X|\geq m/2^{c\sqrt{r}}$ and $|Y|\geq n/2^{c\sqrt{r}}$, finishing the proof.
\end{proof}

\end{document}